\documentclass[11pt,a4paper]{article}

\usepackage[T1]{fontenc}
\usepackage{lmodern}
\usepackage[margin=1.08in]{geometry}
\usepackage{amsmath,amssymb,amsthm,mathtools}
\usepackage{bm}
\usepackage{booktabs}
\usepackage{longtable}
\usepackage{array}
\usepackage{graphicx}
\usepackage{xcolor}

\usepackage{microtype}
\usepackage{enumitem}
\usepackage[hidelinks]{hyperref}
\hypersetup{
 pdftitle={A computer-assisted counterexample to the planar Berenstein conjecture},
 pdfauthor={Matthew J. Colbrook, Siavash Sadeghi, and George Stepaniants},
 pdfsubject={Overdetermined Dirichlet eigenvalue problems and validated numerics},
 pdfkeywords={Berenstein conjecture, overdetermined eigenvalue problem, conformal mapping, disk polynomials, computer-assisted proof},
 bookmarksnumbered=true
}

\newtheorem{theorem}{Theorem}[section]
\newtheorem{proposition}[theorem]{Proposition}
\newtheorem{lemma}[theorem]{Lemma}
\newtheorem{corollary}[theorem]{Corollary}

\newtheorem{remark}[theorem]{Remark}

\newcommand{\D}{\mathbb D}
\newcommand{\T}{\mathbb T}
\newcommand{\R}{\mathbb R}
\newcommand{\C}{\mathbb C}
\newcommand{\PP}{\mathcal P}
\newcommand{\BB}{\mathcal B}
\newcommand{\WW}{\mathcal W}
\newcommand{\YY}{\mathcal Y}
\newcommand{\XX}{\mathcal X}
\newcommand{\ZZ}{\mathcal Z}
\newcommand{\Aop}{\mathcal A}
\newcommand{\KD}{K_{\!D}}
\newcommand{\Ree}{\operatorname{Re}}
\newcommand{\dd}{\,\mathrm d}
\newcommand{\norm}[1]{\left\lVert #1\right\rVert}

\title{A computer-assisted counterexample to the\\planar Berenstein conjecture}
\author{%
 Matthew J. Colbrook\thanks{Department of Applied Mathematics and Theoretical Physics, University of Cambridge, Cambridge, UK. \protect\\Email: \href{mailto:m.colbrook@damtp.cam.ac.uk}{\texttt{m.colbrook@damtp.cam.ac.uk}}.}
 \and
 Siavash Sadeghi\thanks{Department of Mathematics and Statistics, University of Reading, Whiteknights, Reading RG6 6AX, UK. \protect\\Email: \href{mailto:s.sadeghi@pgr.reading.ac.uk}{\texttt{s.sadeghi@pgr.reading.ac.uk}}.}
 \and
 George Stepaniants\thanks{Computing and Mathematical Sciences, California Institute of Technology, Pasadena, CA, USA. \protect\\Email: \href{mailto:gstepan@caltech.edu}{\texttt{gstepan@caltech.edu}}.}%
}
\date{9 August 2026}

\begin{document}
\maketitle

\begin{abstract}
Recent work of Colbrook and Stepaniants produced the first
counterexamples to the planar Pompeiu and Schiffer conjectures and
introduced the conformal fixed-disc, disk-polynomial, and validated-tail
machinery used here. By adapting this framework to the complementary Dirichlet
endpoint, we disprove the unrestricted planar Berenstein conjecture. Specifically, we construct
a bounded simply connected domain $\Omega$ with real-analytic Jordan
boundary, which is not a disc and for which there exist
$$
k\in(27.4381178838,27.4381198839)
$$
and a nonzero real-valued function $u\in C^\omega(\overline\Omega)$ satisfying
$$
 (\Delta+k^2)u=0\quad\text{in }\Omega,
 \qquad u=0,
 \qquad \partial_\nu u=\text{constant}\ne0
 \quad\text{on }\partial\Omega
$$
Thus the overdetermined Dirichlet--Neumann data do not characterize the disc without an additional sign assumption on $u$. The domain has dihedral symmetry of order $26$, but is neither a disc nor centrally symmetric, and the corresponding eigenfunction changes sign. Equivalently, its boundary arclength measure satisfies
$$
\widehat{\sigma_{\partial\Omega}}(k\omega)=0
 \qquad(\omega\in\mathbb S^1).
$$
The proof combines complex-analytic and spectral methods with a computer-assisted existence argument. After conformally transferring the problem to the unit disc, exact support identities and quantitative disk-polynomial estimates yield rigorous control of the infinite-dimensional tail. A Newton--Kantorovich argument then reduces existence to finitely many explicit inequalities, which are certified using interval arithmetic with directed rounding. The extension from the Pompeiu--Schiffer problem is not formal.
The earlier construction absorbs both boundary conditions into a single
inverse-Laplacian equation. At the Dirichlet endpoint considered here,
the nonzero Neumann datum forces the harmonic source modes to remain,
producing a coupled interior--boundary system involving the full
zero-Dirichlet inverse and its Neumann trace, together with a separate
sign-recovery problem.
\end{abstract}



\section{Introduction}

The Berenstein conjecture belongs to the classical Pompeiu--Schiffer circle of rigidity problems, which links rigid-motion integral transforms and Fourier zero sets to overdetermined Laplace eigenvalue problems. Its integral-geometric origins lie in Pompeiu's three papers of 1929~\cite{Pompeiu1929a,Pompeiu1929b,Pompeiu1929c}. Brown, Schreiber, and Taylor placed the problem in a Fourier--spectral framework in 1973, and Williams identified its equivalence, for domains homeomorphic to a ball, with the overdetermined Neumann problem now associated with Schiffer~\cite{BrownSchreiberTaylor1973,Williams1976}. For the Pompeiu--Schiffer problem, Garofalo and Segala used asymptotic methods for complexified Fourier integrals and techniques from univalent-function theory to establish the Pompeiu property for several classes of simply connected planar domains~\cite{GarofaloSegala1991New,GarofaloSegala1991Asymptotic,GarofaloSegala1993,GarofaloSegala1994}. Nigam, Siudeja, and Young combined validated finite-element eigenvalue bounds with a nodal-line argument to prove that the regular pentagon admits a nonconstant Neumann eigenfunction whose boundary trace is strictly positive~\cite{NigamSiudejaYoung2020}. For further history and references, see~\cite{ColbrookStepaniants2026}. These results concern the area measure of the domain, rather than the boundary arclength measure arising in the Berenstein problem.

The Dirichlet problem has an old history. In his 1954 treatment of Rayleigh's fixed-area first Dirichlet eigenvalue problem, Schiffer derived constancy of the normal derivative as the Euler--Lagrange condition at a stationary domain~\cite{Schiffer1954}. A 1957 lecture explicitly asked whether this condition for the first eigenfunction characterises the disc~\cite[p.~132]{Schiffer1957}. The unrestricted rigidity question for an arbitrary, possibly sign-changing, Dirichlet eigenfunction is generally attributed to Berenstein's 1980 paper~\cite{Berenstein1980} and remained open for more than four decades.

Let $\Omega\subset\R^n$ be a bounded domain with smooth boundary. In its modern form, the \emph{Berenstein conjecture} asks whether the existence of $k>0$ and a nonzero function $u$ satisfying
\begin{equation}\label{eq:berenstein-intro}
 (\Delta+k^2)u=0\quad\text{in }\Omega,
 \qquad
 u=0,
 \qquad
 \partial_\nu u=c
 \quad\text{on }\partial\Omega,
 \qquad c\ne0,
\end{equation}
forces $\Omega$ to be a ball. Multiplying $u$ by $c^{-1}$ reduces to $c=1$. Discs satisfy \eqref{eq:berenstein-intro}. If $B_R\subset\R^2$ is the disc of radius $R$ and $j_{0,q}$ is a positive zero of $J_0$, then
$
 u(r)=J_0(j_{0,q}r/R)
$
has zero boundary trace and constant nonzero normal derivative. The conjecture asserts that there are no other such domains. Serrin's moving-plane theorem proves this when the eigenfunction is strictly positive or strictly negative in the interior~\cite{Serrin1971}; consequently any counterexample must involve a higher, sign-changing Dirichlet eigenfunction.

Beyond the sign-definite regime, the principal rigidity results either involve many spectral parameters or require additional hypotheses. Berenstein proved that, in the Euclidean plane, solvability for infinitely many spectral parameters forces the domain to be a disc; Berenstein and Yang extended the corresponding rigidity to the Poincar\'e disc and then to arbitrary Euclidean dimension~\cite{Berenstein1980,BerensteinYang1982,BerensteinYang1987}. Dalmasso later obtained finiteness and nonexistence results for a broader problem with an affine source term, as well as low-eigenvalue rigidity for the Berenstein problem under additional hypotheses~\cite{Dalmasso2010,Dalmasso2014}. Liu obtained a criterion for the Berenstein problem involving the second interior normal derivative, while Kawohl and Lucia characterised related Schiffer-type problems in terms of the third, fourth, and fifth normal derivatives on planar analytic boundaries~\cite{Liu2007,KawohlLucia2020}.

For the linear eigenvalue problem, counterexamples were previously known on unbounded periodic domains~\cite{DaiZhang2023,Minlend2023}. For general semilinear equations, Ruiz showed that the positivity hypothesis in Serrin's theorem is essential by constructing a sign-changing solution on a bounded domain other than a ball for a suitable nonlinearity~\cite{Ruiz2025}. For bounded convex centrally symmetric planar domains with connected $C^{2,\varepsilon}$ boundary, a recent preprint proves rigidity whenever the corresponding eigenvalue exceeds an explicit domain-dependent threshold~\cite{DaiSunWeiZhang2025}. Our example is not centrally symmetric and therefore lies outside that class.

Colbrook and Stepaniants recently gave the first counterexamples to the
planar Pompeiu and Schiffer conjectures~\cite{ColbrookStepaniants2026}.
The conformal fixed-disc, disk-polynomial, and finite-block/infinite-tail
validation framework introduced there provides the conceptual and
computational foundation for the theorem below, which treats the
complementary Dirichlet problem. The extension is not formal: the
harmonic source modes can no longer be eliminated, so the single cubic
equation of~\cite{ColbrookStepaniants2026} is replaced by a coupled
interior--boundary system and a separate sign-recovery argument.
Subsequently, Cao-Labora and de Dios Pont gave an AI-assisted proof of
infinitely many Pompeiu--Schiffer counterexamples~\cite{cao2026counterexamples}.
Taken together, these works provide bounded planar counterexamples to
the complementary Neumann and Dirichlet rigidity conjectures.

Throughout, $C^\omega(\overline\Omega)$ denotes the restrictions to $\overline\Omega$ of real-analytic functions defined on an open neighbourhood of $\overline\Omega$. We also set $\D:=\{z\in\C:|z|<1\}$ and $\T:=\partial\D$.

\begin{theorem}[Main theorem]\label{thm:main}
There exist a bounded simply connected domain $\Omega\subset\R^2$ with real-analytic Jordan boundary, a number
\begin{equation}\label{eq:k-main}
 27.4381178838<k<27.4381198839,
\end{equation}
and a nonzero real-valued function $u\in C^\omega(\overline\Omega)$ such that
\begin{equation}\label{eq:main-pde}
 (\Delta+k^2)u=0\quad\text{in }\Omega,
 \qquad
 u=0,
 \qquad
 \partial_\nu u=1
 \quad\text{on }\partial\Omega.
\end{equation}
The domain is invariant under $D_{13}$, the dihedral group of order $26$, but is neither a disc nor centrally symmetric. Moreover,
\begin{equation}\label{eq:boundary-fourier-main}
 \widehat{\sigma_{\partial\Omega}}(k\omega)
 :=\int_{\partial\Omega}e^{-ik\omega\cdot x}\,\dd s(x)=0
 \qquad\text{for every }\omega\in\mathbb S^1.
\end{equation}
The function $u$ changes sign. Consequently the planar Berenstein conjecture is false.
\end{theorem}

\begin{figure}[t]
 \centering
 \includegraphics[width=.64\textwidth]{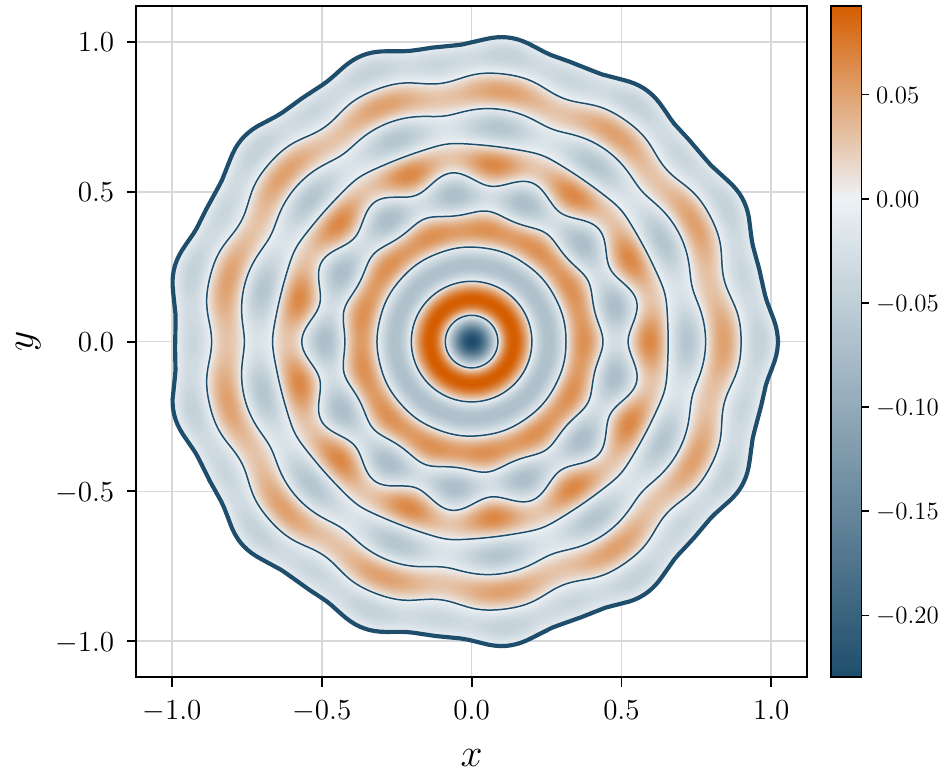}
 \caption{Finite-centre approximation to the sign-changing eigenfunction. The colour represents $u^\circ=(\KD g^\circ)\circ(\phi^\circ)^{-1}$ on $\Omega^\circ=\phi^\circ(\D)$. The eight dark interior curves are the interior zero level sets of $u^\circ$, and the outer curve is $\partial\Omega^\circ$. Under the normalised conformal parametrisations,
$\norm{u_*\circ\phi_*-\KD g^\circ}_{C(\overline\D)}\le2.5\times10^{-7}$,
while $\partial\Omega$ and $\partial\Omega^\circ$ differ pointwise by less than $4.67\times10^{-10}$; see Theorem~\ref{thm:validated-zero}, \eqref{eq:sharp-KD-N}, and \eqref{eq:map-centre-error}. The radial range of the displayed boundary curve is approximately $[0.989729,1.022211]$.}
 \label{fig:domain}
\end{figure}

Figure~\ref{fig:domain} displays the numerical centre used in the proof; estimate \eqref{eq:map-centre-error} controls its distance from the exact boundary. The exact boundary is a regular nodal line of $u$ on which the gradient has constant magnitude:
$$
 \nabla u=\nu,
 \qquad
 |\nabla u|=1
 \quad\text{on }\partial\Omega.
$$
Thus the obstruction is global rather than local. Analytic Cauchy data on a real-analytic curve determine a local two-sided Helmholtz solution, but local solvability does not imply that this solution extends regularly throughout the enclosed domain. The conformal reduction expresses this extension problem as a system of polynomial equations on the unit disc.

\subsection{Boundary arclength and the Fourier circle}

The boundary condition in \eqref{eq:main-pde} has a direct harmonic-analytic reformulation. Let $\sigma_\Gamma$ be arclength measure on a Jordan curve $\Gamma$. We prove in Proposition~\ref{prop:fourier-equivalence} that, for a $C^2$ bounded simply connected domain,
\begin{equation}\label{eq:fourier-equivalence-intro}
 \text{\eqref{eq:berenstein-intro} holds for the frequency $k$}
 \quad\Longleftrightarrow\quad
 \widehat{\sigma_{\partial\Omega}}(k\omega)=0
 \quad(\omega\in\mathbb S^1).
\end{equation}
The forward implication is Green's identity. For the converse, the outgoing single-layer potential with constant density has zero far field, hence vanishes in the exterior by Rellich uniqueness; the normal-derivative jump then produces the required interior eigenfunction. Thus Theorem~\ref{thm:main} also gives a noncircular analytic Jordan curve whose arclength Fourier transform vanishes on the circle $\{\xi\in\R^2:|\xi|=k\}$.

The Fourier equations provide an effective search mechanism, but they are not well suited to uniform tail estimates. The proof therefore proceeds through the fixed-disc system, whose high-mode linearisation has a diagonal, isometric principal part and admits explicit monotone bounds.

\subsection{The Pompeiu–Schiffer framework and the new Dirichlet endpoint}

The two constructions arise at complementary endpoints of mixed
overdetermined Helmholtz problems. Wheeler constructed local
noncircular branches with both boundary constants nonzero~\cite{Wheeler2025}.
Colbrook and Stepaniants worked in the $10$-fold symmetry class and
continued to the endpoint at which the Neumann datum vanishes, yielding
the Pompeiu--Schiffer counterexample~\cite{ColbrookStepaniants2026}.
The present $D_{13}$ construction instead reaches the endpoint at which
the Dirichlet datum vanishes. At the Fourier level,
\cite{ColbrookStepaniants2026} concerns cancellation of the area measure
$\mathbf{1}_{\Omega}$, whereas the present problem concerns cancellation
of the boundary arclength measure $\sigma_{\partial\Omega}$.

The distinction becomes more pronounced after conformal transfer to the disc. In the Schiffer problem the decomposition $U=1+v$ gives both $v|_{\T}=0$ and $\left.\partial_r v\right|_{\T}=0$. Consequently $g=\Delta v$ lies in the range compatible with both vanishing traces, its harmonic coordinates disappear, and an inverse $K$ of the Laplacian on this range absorbs the boundary conditions. With $p_{\mathrm{S}}:=k\phi'$ as in~\cite{ColbrookStepaniants2026}, the problem reduces to the single cubic equation
$$
 g+|p_{\mathrm{S}}|^2(1+Kg)=0.
$$
In the present problem, $U|_{\T}=0$ but $\left.\partial_r U\right|_{\T}=|a|\neq0$. The harmonic source coordinates must therefore be retained, and the Neumann condition cannot be incorporated into the inverse. The full zero-Dirichlet inverse $\KD$ and its trace $N$ instead yield the coupled interior--boundary system
\begin{equation}\label{eq:fourier-equivalence}
 g+k^2|a|^2\KD g=0,
 \qquad
 (Ng)^2-|a|^2=0.
\end{equation}
The squared boundary equation requires a separate sign argument to recover $Ng=|a|$, as well as estimates for the source and shape tails. For the present $D_{13}$ candidate, the trace operator annihilates every radial mode $s\ge1$, and the high angular modes have disjoint support from the finite block. The resulting finite block has dimension $882$, compared with $2471$ in the Pompeiu--Schiffer proof.

\subsection{Outline of the proof}

The argument proceeds through three analytic stages, followed by
reconstruction. After establishing~\eqref{eq:fourier-equivalence}, we
impose simultaneously the conformal-radius normalisation
$\phi'(0)=1$ and the flux normalisation $\partial_\nu u=1$.
Writing
$$
    a=\phi',\qquad U=u\circ\phi,\qquad g=\Delta U,
$$
gives the coupled interior--boundary system described in
Section~1.2, with
$$
    U=K_Dg,
    \qquad
    Ng=\left.\partial_r(K_Dg)\right|_{\mathbb T}.
$$
Positivity of $Ng$ recovers the unsquared condition $Ng=|a|$.

The coefficient space is built from $D_{13}$-adapted disk polynomials. Their positive product formula makes the weighted $\ell^1$ coefficient space a Banach algebra with multiplication constant one. We derive an explicit three-term formula for $\KD$ and the trace formula
$$
 N\Phi_{\ell,0}=\frac{e^{13i\ell\theta}}{2(13|\ell|+1)},
 \qquad N\Phi_{\ell,s}=0\quad(s\ge1),
$$
with sharp operator norms $\norm{\KD}=1/4$ and $\norm{N}=1/2$.

An a posteriori contraction theorem is then applied near a finitely supported centre. The approximate inverse is represented by a dense matrix on the retained block and by signed coordinate isometries on the tails. Support calculations reduce the retained-block contributions to $11{,}521$ source-tail columns and forty shape-tail columns; explicit bounds, monotone in the indices, cover all remaining columns. Interval arithmetic with directed rounding verifies
$$
 Y<3.63\times10^{-9},
 \qquad
 Z<0.474,
$$
and a strict contraction at radius $10^{-6}$. Finally, coefficient inequalities give $\Ree\phi'>0$, positivity of $Ng$, analytic boundary regularity, and noncircularity.

\section{Fourier characterisation and conformal reduction}

\subsection{Fourier-circle equivalence}

We use the Fourier convention
$$
 \widehat\mu(\xi)=\int_{\R^2}e^{-i\xi\cdot x}\,\dd\mu(x)
$$
for finite measures. The outgoing fundamental solution is
$$
 G_k(x)=\frac{i}{4}H_0^{(1)}(k|x|).
$$
With this convention,
$$
 (\Delta+k^2)G_k=-\delta_0.
$$
In this subsection, $\nu$ is the outward unit normal of $\Omega$; $\gamma_0^-$ and $\gamma_0^+$ are the interior and exterior Dirichlet traces, respectively, and $\gamma_1^-$ and $\gamma_1^+$ are the corresponding normal traces, all taken in the direction $\nu$.

\begin{proposition}[Fourier-circle equivalence]\label{prop:fourier-equivalence}
Let $\Omega\subset\R^2$ be a bounded domain whose boundary $\Gamma$ is a connected $C^2$ Jordan curve, and let $k>0$. The following are equivalent.
\begin{enumerate}[label=(\roman*)]
\item There exist a nonzero real-valued function $u\in H^2(\Omega)$ and a constant $c\ne0$ such that
$$
 (\Delta+k^2)u=0\quad\text{in }\Omega,
 \qquad
 u=0,
 \qquad
 \partial_\nu u=c
 \quad\text{on }\Gamma.
$$
\item The arclength measure of $\Gamma$ satisfies
$
 \widehat{\sigma_\Gamma}(k\omega)=0$ for every $\omega\in\mathbb S^1$.
\end{enumerate}
\end{proposition}

\begin{proof}
Assuming (i), since both $u$ and $v_\omega(x):=e^{-ik\omega\cdot x}$ solve the Helmholtz equation, Green's second identity gives
$$
 0=\int_\Gamma
 \bigl(u\,\partial_\nu v_\omega-v_\omega\,\partial_\nu u\bigr)\,\dd s
 =-c\int_\Gamma e^{-ik\omega\cdot x}\,\dd s(x),
$$
which proves (ii). Conversely, suppose (ii) holds. The outgoing single-layer potential with constant density is
$$
 w(x)=\int_\Gamma G_k(x-y)\,\dd s(y).
$$
Its far-field pattern is a nonzero constant multiple of $\widehat{\sigma_\Gamma}(k\omega)$, so it vanishes identically. Rellich's lemma and exterior uniqueness imply that $w=0$ in the unbounded component of $\R^2\setminus\overline\Omega$, \cite[Chs.~2--3]{ColtonKress2019}. The single-layer potential is continuous across $\Gamma$, and hence
$$
 \gamma_0^-w=\gamma_0^+w=0.
$$
For the convention above, the normal traces of the single-layer operator
$$
 S_k\varphi(x):=\int_\Gamma G_k(x-y)\varphi(y)\,\dd s(y)
$$
satisfy the equation
$$
 \gamma_1^-S_k\varphi-\gamma_1^+S_k\varphi=\varphi.
$$
See \cite[Thm.~6.11]{McLean2000} for these continuity, mapping, and jump properties. Since the exterior field vanishes, $\gamma_1^+w=0$, and therefore
$
 \gamma_1^-w=1.
$
Standard mapping properties of the single-layer potential give $w|_\Omega\in H^1(\Omega)$, and $w|_\Omega$ is a weak solution of
$$
 (\Delta+k^2)w=0\quad\text{in }\Omega
$$
with zero Dirichlet trace. Zero-Dirichlet elliptic regularity on the $C^2$ domain $\Omega$ therefore gives $w|_\Omega\in H^2(\Omega)$. Taking the real part preserves the zero Dirichlet trace and the unit Neumann trace. Thus $\Ree(w|_\Omega)$ is a nontrivial real-valued function in $H^2(\Omega)$, and (i) holds with $c=1$.
\end{proof}

\subsection{Normalisation and transfer to the unit disc}

We identify $\R^2$ with $\C$. The conformal radius and the constant flux may be normalised simultaneously. Choose any point $p\in\Omega$, translate it to the origin, and let $\psi_0:\D\to\Omega-p$ be a Riemann map satisfying $\psi_0(0)=0$. Precomposing $\psi_0$ by the rotation
$
 z\mapsto e^{-i\arg\psi_0'(0)}z
$
gives a Riemann map $\psi$ for which
$$
 \psi'(0)=c_\phi:=|\psi_0'(0)|>0.
$$
Translations and rotations of the physical plane preserve the overdetermined problem; on the Fourier side a translation contributes only a unimodular phase and a rotation reparametrises $\omega$. Positive dilation preserves the problem after the corresponding rescaling of the wavenumber.

Suppose that $\partial_\nu u=c_0\ne0$. On
$$
 \widetilde\Omega=\frac{\Omega-p}{c_\phi}
$$
the rescaled function, wavenumber, and Riemann map are
$$
 \widetilde u(x)=\frac{u(p+c_\phi x)}{c_\phi c_0},
 \qquad
 \widetilde k=c_\phi k,
 \qquad
 \phi(z)=\frac{\psi(z)}{c_\phi}.
$$
Then $\widetilde u$ satisfies the overdetermined problem on $\widetilde\Omega$ with wavenumber $\widetilde k$ and outward normal derivative one. The map $\phi:\D\to\widetilde\Omega$ satisfies $\phi(0)=0$ and $\phi'(0)=1$. Moreover,
$$
 \widehat{\sigma_{\partial\widetilde\Omega}}(\widetilde k\omega)
 =c_\phi^{-1}e^{ik\omega\cdot p}
 \widehat{\sigma_{\partial\Omega}}(k\omega),
$$
so the Fourier-circle condition is preserved. After relabelling the normalised objects, we may therefore assume
\begin{equation}\label{eq:normalisation}
 \phi(0)=0,
 \qquad
 \phi'(0)=1,
 \qquad
 \partial_\nu u=1.
\end{equation}
We use the normalised conformal parametrisation of \cite[Sec.~2.1]{ColbrookStepaniants2026}, now in the real reflection-symmetric $D_{13}$ class. Thus, with $m=13$, we impose
\begin{equation}\label{eq:a-series}
 a(z):=\phi'(z)=1+\sum_{j\ge1}a_jz^{mj},
 \qquad a_j\in\R.
\end{equation}
Then
\begin{equation}\label{eq:phi-reconstruction}
 \phi(z)=z+\sum_{j\ge1}\frac{a_j}{mj+1}z^{mj+1}.
\end{equation}
The reality of the coefficients gives reflection symmetry, while the exponents imply
$$
 \phi(e^{2\pi i/m}z)=e^{2\pi i/m}\phi(z).
$$
Thus a univalent map of the form \eqref{eq:phi-reconstruction} has $D_m$-symmetric image.
For the pullback $U:=u\circ\phi$, the conformal transformation identities in \cite[eq.~(6)]{ColbrookStepaniants2026} give
\begin{equation}\label{eq:fixed-disc-pde}
 \Delta U+k^2|a|^2U=0\quad\text{in }\D,
 \qquad
 U=0,
 \qquad
 \partial_r U=|a|
 \quad\text{on }\T.
\end{equation}
Let $\KD$ be the zero-Dirichlet inverse developed in Section~\ref{sec:disk-algebra}. Its Neumann trace is
$$
 Ng:=\left.\partial_r(\KD g)\right|_{\T}.
$$
The identities $g=\Delta U$ and $U=\KD g$ transform \eqref{eq:fixed-disc-pde} into the polynomial system
\begin{equation}\label{eq:F-system}
 \begin{aligned}
 F_1(g,k,a)&:=g+k^2|a|^2\KD g=0&&\text{in }\D,\\
 F_2(g,k,a)&:=(Ng)^2-|a|^2=0&&\text{on }\T.
 \end{aligned}
\end{equation}
Zeros of \eqref{eq:F-system} recover overdetermined eigenfunctions once the mapping properties of $\KD$ and $N$, and the sign lost by squaring the boundary equation, have been resolved. The next section develops the required coefficient theory and then proves the equivalence with the original boundary problem.

\section{Disk polynomials, the Dirichlet inverse, and the trace}\label{sec:disk-algebra}

The fixed-disc system admits a quantitative polynomial formulation once multiplication, the zero-Dirichlet inverse, and its Neumann trace are controlled on a common coefficient space. The disk-polynomial basis supplies these three ingredients.

\subsection{The positive coefficient algebra}

Following \cite[Sec.~2.2]{ColbrookStepaniants2026}, for $z=re^{i\theta}$, $\ell\in\mathbb Z$, and $s\ge0$, the $D_{13}$-adapted disk polynomials are
\begin{equation}\label{eq:disk-polynomials}
 \Phi_{\ell,s}(r,\theta)
 =r^{m|\ell|}P_s^{(0,m|\ell|)}(2r^2-1)e^{im\ell\theta},
 \qquad m=13.
\end{equation}
Here $P_s^{(\alpha,\beta)}$ is the Jacobi polynomial normalised by $P_s^{(\alpha,\beta)}(1)=\binom{s+\alpha}{s}$; see \cite[Ch.~IV, eqs.~(4.1.1) and (4.3.3)]{Szego1975} for the normalisation and Jacobi orthogonality conventions. Jacobi orthogonality, with $\dd A$ denoting planar area measure, gives
\begin{equation}\label{eq:disk-orthogonality}
 \int_\D \Phi_{\ell,s}\overline{\Phi_{\ell',t}}\,\dd A
 =\frac{\pi}{m|\ell|+2s+1}\,\delta_{\ell\ell'}\delta_{st}.
\end{equation}
The family is complete in the $m$-fold rotationally invariant subspace of $L^2(\D)$. Moreover, $\Phi_{\ell,0}$ is harmonic, equal to $z^{m\ell}$ for $\ell\ge0$ and $\overline z^{m|\ell|}$ for $\ell<0$. A real-valued reflection-symmetric function has the unique expansion
\begin{equation}\label{eq:real-expansion}
 f=\sum_{s\ge0}f_{0,s}\Phi_{0,s}
 +\sum_{\ell\ge1}\sum_{s\ge0}f_{\ell,s}
 (\Phi_{\ell,s}+\Phi_{-\ell,s}),
 \qquad f_{\ell,s}\in\R.
\end{equation}
The symmetry multiplicities are $c_0=1$ and $c_\ell=2$ for $\ell\ge1$.

The positive linearisation theorem \cite[Lemma~2.2]{ColbrookStepaniants2026}, which is independent of the symmetry order, applies with $m=13$. Products of basis elements have finite expansions with nonnegative coefficients of total mass one. The same result gives
$$
 \overline{\Phi_{\ell,s}}=\Phi_{-\ell,s},
 \qquad
 \sup_{\overline\D}|\Phi_{\ell,s}|\le1.
$$

We use the $m=13$ versions of the weighted disk and holomorphic coefficient spaces from \cite[Secs.~2.2--2.3]{ColbrookStepaniants2026}, together with a boundary Wiener algebra for the Neumann trace. For $\rho>1$, define the complex coefficient space
\begin{equation}\label{eq:W-norm}
 \WW_\rho
 =\left\{f=\sum_{\ell\in\mathbb Z}\sum_{s\ge0}f_{\ell,s}\Phi_{\ell,s}:
 f_{\ell,s}\in\C,\quad
 \norm{f}_\rho:=\sum_{\ell\in\mathbb Z}\sum_{s\ge0}\rho^{|\ell|}|f_{\ell,s}|<\infty\right\}.
\end{equation}
For the real symmetry class \eqref{eq:real-expansion}, this norm is
$$
 \norm{f}_\rho
 =\sum_{\ell\ge0}\sum_{s\ge0}c_\ell\rho^\ell|f_{\ell,s}|.
$$
The closed real subspace of $\WW_\rho$ consisting of these real-valued, reflection-symmetric functions is denoted by $\YY_\rho$. The real boundary Wiener algebra $\BB_\rho$ carries the norm
$$
 \norm{b}_{\BB_\rho}
 =\sum_{\ell\ge0}c_\ell\rho^\ell|b_\ell|
$$
for
$$
 b(\theta)=b_0+\sum_{\ell\ge1}b_\ell(e^{im\ell\theta}+e^{-im\ell\theta}),
 \qquad b_\ell\in\R.
$$
Finally, let
\begin{equation}\label{eq:P-space}
 \PP_\rho
 =\left\{a(z)=\sum_{j\ge0}a_jz^{mj}:
 a_j\in\R,\quad
 \norm{a}_{\PP_\rho}:=\sum_{j\ge0}\rho^j|a_j|<\infty\right\}.
\end{equation}
Every $a\in\PP_\rho$ is holomorphic for $|z|<\rho^{1/m}$.

\begin{corollary}[Banach-algebra bounds]\label{cor:coefficient-algebra}
The space $\WW_\rho$ is a Banach algebra, and its norm is submultiplicative:
\begin{equation}\label{eq:algebra-bound}
 \norm{f_1f_2}_\rho\le\norm{f_1}_\rho\norm{f_2}_\rho.
\end{equation}
The space $\YY_\rho$ is a closed real subalgebra of $\WW_\rho$, and both spaces embed continuously in $C(\overline\D)$. The spaces $\BB_\rho$ and $\PP_\rho$ are Banach algebras with submultiplicative norms and embed continuously in $C(\T)$ and $C(\overline\D)$, respectively. Moreover, for $a,\eta\in\PP_\rho$ and $f\in\YY_\rho$,
\begin{align*}
 \norm{|a|^2f}_\rho
 &\le\norm{a}_{\PP_\rho}^2\norm{f}_\rho,\\
 \norm{(\overline a\eta+a\overline\eta)f}_\rho
 &\le2\norm{a}_{\PP_\rho}\norm{\eta}_{\PP_\rho}\norm{f}_\rho.
\end{align*}
On $\T$, we have that
$$
 \norm{|a|^2}_{\BB_\rho}
 \le\norm{a}_{\PP_\rho}^2,\qquad
 \norm{\overline a\eta+a\overline\eta}_{\BB_\rho}
 \le2\norm{a}_{\PP_\rho}\norm{\eta}_{\PP_\rho}.
$$
\end{corollary}

\begin{proof}
The assertions for $\WW_\rho$, including \eqref{eq:algebra-bound} and the embedding into $C(\overline\D)$, follow from \cite[Corollary~2.3]{ColbrookStepaniants2026}; the proof there is independent of the symmetry order. The conjugation symmetry defining $\YY_\rho$ is preserved by multiplication, so $\YY_\rho$ is a closed real subalgebra. The assertions for $\BB_\rho$ and $\PP_\rho$ follow from weighted Fourier convolution. The disk estimate $\norm{|a|^2}_\rho\le\norm{a}_{\PP_\rho}^2$ follows from \cite[eq.~(23)]{ColbrookStepaniants2026}; submultiplicativity and polarisation give the remaining disk estimates. The same convolution argument gives the boundary estimates.
\end{proof}

\subsection{An explicit zero-Dirichlet inverse}

The preceding algebra controls the nonlinear factors in $F$. The remaining ingredient is an exact coefficient representation of the linear map $g\mapsto(\KD g,Ng)$. Because the Neumann trace is nonzero, the harmonic source coordinates $s=0$ must be retained. This is the main distinction from the Schiffer problem in~\cite{ColbrookStepaniants2026}, where the inverse acts on sources compatible with both zero Dirichlet and zero Neumann traces.

For $n=m|\ell|$, the action of $\KD$ on harmonic modes is
\begin{equation}\label{eq:KD-harmonic}
 \KD\Phi_{\ell,0}
 =\frac{\Phi_{\ell,1}-\Phi_{\ell,0}}{4(n+1)(n+2)}.
\end{equation}
For $s\ge1$, we use the compatible inverse formula from \cite[eq.~(19)]{ColbrookStepaniants2026}; only the harmonic columns $s=0$ are new at the Dirichlet endpoint. With $D:=n+2s$, its action is
\begin{equation}\label{eq:KD-nonharmonic}
 \KD\Phi_{\ell,s}
 =\frac{\Phi_{\ell,s-1}}{4D(D+1)}
 -\frac{\Phi_{\ell,s}}{2D(D+2)}
 +\frac{\Phi_{\ell,s+1}}{4(D+1)(D+2)}.
\end{equation}

\begin{lemma}[Exact inverse and Neumann trace]\label{lem:KD-exact}
For every $\ell\in\mathbb Z$ and $s\ge0$,
$$
 \Delta\KD\Phi_{\ell,s}=\Phi_{\ell,s},
 \qquad
 (\KD\Phi_{\ell,s})|_{\T}=0.
$$
Moreover,
\begin{equation}\label{eq:N-formula}
 N\Phi_{\ell,0}=\frac{e^{im\ell\theta}}{2(m|\ell|+1)},
 \qquad
 N\Phi_{\ell,s}=0\quad(s\ge1).
\end{equation}
\end{lemma}

\begin{proof}
For $s\ge1$, the inverse identity and both trace cancellations follow from \cite[Lemma~2.4]{ColbrookStepaniants2026}; the proof there depends only on the angular degree and applies with $n=m|\ell|$. For $s=0$, the identity
$$
 P_1^{(0,n)}(2r^2-1)=(n+2)r^2-(n+1)
$$
reduces \eqref{eq:KD-harmonic} to
$$
 \KD\Phi_{\ell,0}
 =\frac{r^n(r^2-1)e^{im\ell\theta}}{4(n+1)}.
$$
Direct differentiation gives $\Delta\KD\Phi_{\ell,0}=\Phi_{\ell,0}$, its zero Dirichlet trace, and the first formula in \eqref{eq:N-formula}.
\end{proof}

\begin{lemma}[Sharp operator bounds]\label{lem:KD-norms}
The maps
$$
 \KD:\YY_\rho\to\YY_\rho,
 \qquad
 N:\YY_\rho\to\BB_\rho
$$
are bounded, with
\begin{equation}\label{eq:sharp-KD-N}
\norm{\KD}=\frac14,
 \qquad
 \norm{N}=\frac12.
\end{equation}
For each basis vector, with $n=m|\ell|$, the $\KD$ column norm is
\begin{equation}\label{eq:kappa-column}
 \kappa_{\ell,0}=\frac1{2(n+1)(n+2)},
 \qquad
 \kappa_{\ell,s}=\frac1{(n+2s)(n+2s+2)}\quad(s\ge1).
\end{equation}
\end{lemma}

\begin{proof}
For $s\ge1$, the column-sum calculation in \cite[eq.~(21)]{ColbrookStepaniants2026}, with $D=n+2s$, gives $\kappa_{\ell,s}=1/[D(D+2)]$. For $s=0$, the absolute column sum in \eqref{eq:KD-harmonic} is $1/[2(n+1)(n+2)]$, whose maximum is $1/4$. Thus $\norm{\KD}=1/4$, with equality on $\Phi_{0,0}$. Formula \eqref{eq:N-formula} gives column ratio $1/[2(n+1)]$ and hence $\norm N=1/2$, again with equality at $\ell=0$.
\end{proof}

The next proposition identifies the coefficient formulae with the Sobolev solution operator used in the conformal reduction.

\begin{proposition}[Dirichlet solution operator]\label{prop:weak-KD}
For every $q>2$ and $g\in\YY_\rho$, the coefficient series $U=\KD g$ is the unique function in $W^{2,q}(\D)\cap W^{1,q}_0(\D)$ satisfying
$$
 \Delta U=g\quad\text{in }\D.
$$
Its outward normal trace is
$
 \gamma_1 U:=\left.\partial_r U\right|_{\T}=Ng.
$
\end{proposition}

\begin{proof}
The truncation and Dirichlet-regularity argument in \cite[proof of Lemma~2.5]{ColbrookStepaniants2026} applies to $\KD:\YY_\rho\to\YY_\rho$, by Lemmas~\ref{lem:KD-exact} and~\ref{lem:KD-norms}. It gives the unique solution $U=\KD g$ in $W^{2,q}(\D)\cap W^{1,q}_0(\D)$. For finite coefficient truncations $g_n$, continuity of the normal-trace map, the identities $\gamma_1\KD g_n=Ng_n$, and the bounded embedding $\BB_\rho\hookrightarrow C(\T)$ give $\gamma_1U=Ng$.
\end{proof}

\subsection{The fixed-disc equivalence and analytic operator}

The preceding coefficient formulae and mapping properties now justify the fixed-disc reduction.

\begin{proposition}[Equivalence of the fixed-disc system]\label{prop:fixed-disc-equivalence}
Fix $\rho>1$ and $m=13$. Suppose $k>0$,
$$
 g\in\YY_\rho,
 \qquad
 a(z)=1+\sum_{j\ge1}a_jz^{mj}\in\PP_\rho,
 \qquad a_j\in\R,
$$
where $\YY_\rho$ is the real-valued, reflection-symmetric $D_m$ coefficient space introduced above. Assume in addition that the map \eqref{eq:phi-reconstruction} is univalent on a neighbourhood of $\overline\D$ and that $a$ does not vanish on $\T$. If $(g,k,a)$ is a zero of \eqref{eq:F-system}, then there is a unique sign $\varepsilon\in\{-1,1\}$ such that
$
 N(\varepsilon g)=|a|\quad\text{on }\T,
$
and
$
 u=\bigl(\KD(\varepsilon g)\bigr)\circ\phi^{-1}
$
is a nonzero real-valued solution of \eqref{eq:main-pde} on $\Omega=\phi(\D)$. Conversely, let $u\in H^2(\Omega)$ be a normalised real-valued solution of \eqref{eq:main-pde} whose normalised Riemann map has the displayed real $D_m$ form, with $a=\phi'\in\PP_\rho$. If the pulled-back source
$
 g:=\Delta(u\circ\phi)
$
belongs to $\YY_\rho$, then $(g,k,a)$ is a zero of \eqref{eq:F-system}.
\end{proposition}

\begin{proof}
The second equation in \eqref{eq:F-system} says
$
 (Ng)^2=|a|^2.
$
By \eqref{eq:N-formula}, Lemma~\ref{lem:KD-norms}, and Corollary~\ref{cor:coefficient-algebra}, $Ng$ is a continuous real-valued function. Since $a$ does not vanish on $\T$, $Ng$ never vanishes and has constant sign on the connected circle. Moreover,
$$
 F_1(-g,k,a)=-F_1(g,k,a),
 \qquad
 F_2(-g,k,a)=F_2(g,k,a).
$$
Thus the zero set of $F$ is invariant under $g\mapsto-g$, and there is a unique $\varepsilon\in\{-1,1\}$ such that $N(\varepsilon g)=|a|$.

For $U:=\KD(\varepsilon g)$, the triple $(\varepsilon g,k,a)$ is again a zero, so the first equation in \eqref{eq:F-system} gives the interior equation in \eqref{eq:fixed-disc-pde}, and the definition of $\KD$ gives $U|_{\T}=0$. For $u=U\circ\phi^{-1}$, the Sobolev transfer in \cite[proof of Proposition~2.1]{ColbrookStepaniants2026} and the normal-derivative identity in \cite[eq.~(6)]{ColbrookStepaniants2026} give
$$
 \partial_r U=|a|\,(\partial_\nu u)\circ\phi
 \quad\text{on }\T.
$$
Hence $\partial_\nu u=1$. The nonzero boundary trace $\partial_r U=|a|$ also shows that $U$, and therefore $u$, is nontrivial.

Conversely, a normalised solution in the stated class determines $U:=u\circ\phi$ and $g:=\Delta U$. Since $U|_{\T}=0$, uniqueness of the weak zero-Dirichlet problem gives $U=\KD g$. The same conformal transformation identities yield
$$
 g+k^2|a|^2\KD g=0,
 \qquad
 Ng=\left.\partial_r U\right|_{\T}=|a|,
$$
and hence both equations in \eqref{eq:F-system}.
\end{proof}

We now formulate this equivalence as a polynomial equation on weighted coefficient spaces. Let $\PP_{\rho,0}$ be the subspace of $\PP_\rho$ consisting of perturbations with zero constant coefficient. The unknown is
$$
 x=(g,k,\widetilde a)\in\XX_\rho:=\YY_\rho\times\R\times\PP_{\rho,0},
 \qquad a=1+\widetilde a,
$$
and the residual belongs to
$$
 \ZZ_\rho:=\YY_\rho\times\BB_\rho.
$$
The map $F=(F_1,F_2)$ from \eqref{eq:F-system} is a real polynomial map $\XX_\rho\to\ZZ_\rho$. In what follows, $(g,k,a)$ denotes the same element, where $a=1+\widetilde a$ is the full conformal derivative; its third coordinate in $\XX_\rho$ is therefore $a-1$. At $x=(g,k,a)$, let
$$
 U=\KD g,
 \qquad b_g=Ng.
$$
For a variation $h=(\dot g,\dot k,\dot a)\in\XX_\rho$,
\begin{align}
 DF_1(x)h
 &=\dot g+k^2|a|^2\KD\dot g\notag\\
 &\quad+\bigl(2k\dot k|a|^2
 +k^2(\overline a\dot a+a\overline{\dot a})\bigr)U,
 \label{eq:DF1}\\
 DF_2(x)h
 &=2b_gN\dot g-(\overline a\dot a+a\overline{\dot a}).
 \label{eq:DF2}
\end{align}
Suppressing the fixed parameter $\rho$ from the norm subscripts, we use the scaled norms
\begin{align}
 \norm{(\dot g,\dot k,\dot a)}_\XX
 &=\norm{\dot g}_\rho+\tau|\dot k|+2\sigma\norm{\dot a}_{\PP_\rho},
 \label{eq:X-scaled}\\
 \norm{(y,b)}_\ZZ
 &=\norm{y}_\rho+\sigma\norm{b}_{\BB_\rho}.
 \label{eq:Z-scaled}
\end{align}
The proof uses
\begin{equation}\label{eq:scales}
 \rho=\frac{2296835809958953}{2251799813685248}\approx1.02,
 \qquad
 \sigma=75,
 \qquad
 \tau=1.
\end{equation}
The factor $2\sigma$ makes the leading term in each high-index shape direction in \eqref{eq:DF2} a signed coordinate isometry. With this scaling, the approximate inverse can act densely on a finite block and by signed coordinate isometries on the tails. This structure isolates the finite inverse calculation and reduces the tail analysis to the coupling between these pieces.

\section{A posteriori validation}\label{sec:validation}

The estimates of Section~\ref{sec:disk-algebra} reduce existence to the injectivity of an approximate inverse, a residual bound at the centre, and a bound for the derivative defect. The validation begins with the abstract contraction criterion and then verifies its hypotheses at the numerical centre.

\subsection{A radii-polynomial theorem}

Fix a finitely supported centre $x^\circ=(g^\circ,k^\circ,a^\circ)$ and a bounded linear approximate inverse $\Aop:\ZZ_\rho\to\XX_\rho$. The associated quantities are
$$
 A_0=\norm{a^\circ}_{\PP_\rho},
 \qquad
 U_0=\norm{\KD g^\circ}_\rho,
 \qquad
 K_0=|k^\circ|,
 \qquad
 \Lambda_A=\norm{\Aop}.
$$
The following criterion is a radii-polynomial form of the Newton--Kantorovich argument, introduced in~\cite{DayLessardMischaikow2007} and adapted to analytic weighted sequence spaces in~\cite[Sec.~3, Def.~1 and Prop.~2]{HungriaLessardMirelesJames2016}.

\begin{theorem}[Radii criterion]\label{thm:radii}
Suppose $\Aop$ is injective and
\begin{equation}\label{eq:YZ-def}
 \norm{\Aop F(x^\circ)}_\XX\le Y,
 \qquad
 \norm{I-\Aop DF(x^\circ)}_{\XX\to\XX}\le Z.
\end{equation}
The quantities entering the criterion are
\begin{align}
 M_1(r)
 &=\left(K_0+\frac r\tau\right)^2
 \left(A_0+\frac r{2\sigma}\right)^2
 \left(U_0+\frac r4\right)-K_0^2A_0^2U_0-r\left(
 \frac{2K_0A_0^2U_0}{\tau}
 +\frac{K_0^2A_0U_0}{\sigma}
 +\frac{K_0^2A_0^2}{4}
 \right),
 \label{eq:M1}\\
 M_2(r)&=\left(\frac14+\frac1{4\sigma^2}\right)r^2,
 \label{eq:M2}\\
 \mathcal R(r)&=\Lambda_A\bigl(M_1(r)+\sigma M_2(r)\bigr),
 \label{eq:R-majorant}\\
 p(r)&=Y+(Z-1)r+\mathcal R(r).
 \label{eq:radii-polynomial}
\end{align}
If, for some $r>0$,
\begin{equation}\label{eq:radii-conditions}
 p(r)<0,
 \qquad
 Z+\mathcal R'(r)<1,
\end{equation}
then $F$ has a unique zero in the closed ball $\overline B_\XX(x^\circ,r)$.
\end{theorem}

\begin{proof}
For the Newton map $T(x):=x-\Aop F(x)$ and $\norm{h}_\XX\le r$,
$$
 T(x^\circ+h)-x^\circ
 =-\Aop F(x^\circ)+(I-\Aop DF(x^\circ))h-\Aop\mathcal N(h),
$$
where $\mathcal N$ is the nonlinear Taylor remainder. The algebra inequality, Lemma~\ref{lem:KD-norms}, and
$$
 \norm{\dot g}_\rho\le r,
 \qquad
 |\dot k|\le r/\tau,
 \qquad
 \norm{\dot a}_{\PP_\rho}\le r/(2\sigma)
$$
show that $\norm{\mathcal N_1(h)}_\rho\le M_1(r)$. This follows by expanding
$$
 \left(K_0+\frac r\tau\right)^2
 \left(A_0+\frac r{2\sigma}\right)^2
 \left(U_0+\frac r4\right).
$$
Every coefficient is nonnegative, and the constant and degree-one terms are precisely those subtracted in \eqref{eq:M1}. For the boundary equation, the exact nonlinear remainder is
$
 (N\dot g)^2-|\dot a|^2,
$
so we have that
$$
 \norm{\mathcal N_2(h)}_{\BB_\rho}
 \le\left(\norm{N}^2+\frac1{4\sigma^2}\right)r^2=M_2(r).
$$
Thus $\mathcal R(r)$ bounds the nonlinear contribution after applying $\Aop$. The first inequality in \eqref{eq:radii-conditions} makes $T$ map the closed ball strictly into itself. Since the scalar majorant has nonnegative coefficients, $\mathcal R'(r)$ bounds the derivative variation throughout the ball; the second inequality makes $T$ a strict contraction. Banach's theorem gives a unique fixed point $x_*$. At the fixed point $\Aop F(x_*)=0$, and injectivity of $\Aop$ gives $F(x_*)=0$.
\end{proof}

It remains to construct $x^\circ$ and $\Aop$, prove the injectivity of $\Aop$, and bound $Y$ and $Z$.

\subsection{The numerical centre and finite block}\label{subsec:centre}

The candidate was found by following a $D_{13}$-symmetric branch of solutions with constant Dirichlet and Neumann data. Among the radial solutions normalised by unit flux,
$$
 U_\mu(r)=-\frac{J_0(\mu r)}{\mu J_1(\mu)},
 \qquad
 U_\mu(1)=-\frac{J_0(\mu)}{\mu J_1(\mu)}.
$$
The radial linearisation in \cite[Sec.~3.2]{ColbrookStepaniants2026}, which is unchanged by the unit-flux normalisation, shows that the angular sector $\cos(m\theta)$ becomes singular at zeros of
\begin{equation}\label{eq:wronskian}
 W_{1,m}(\mu)=J_1(\mu)J_m'(\mu)-J_m(\mu)J_1'(\mu).
\end{equation}
The continuation began near the fifth positive zero of $W_{1,13}$,
$
 \mu\approx27.3968010178226,
$
and reached a point at which the Dirichlet datum vanished. The Fourier-circle equivalence gives a second formulation of the numerical search.

\begin{remark}[Fourier moment equations]\label{rem:moment-search}
If a $D_m$-symmetric curve is parametrised in polar form by
$$
 \Gamma(\theta)=\mathfrak r(m\theta)e^{i\theta},
 \qquad
 \mathfrak r(\vartheta)=1+\sum_{j\ge1}\mathfrak r_j\cos(j\vartheta),
$$
then Proposition~\ref{prop:fourier-equivalence} reduces the search to the scalar moment equations
$$
 \frac1{2\pi}\int_0^{2\pi}
 \sqrt{\mathfrak r(\vartheta)^2+m^2\mathfrak r'(\vartheta)^2}\,
 J_{m\ell}\bigl(k\mathfrak r(\vartheta)\bigr)\cos(\ell\vartheta)\,\dd\vartheta=0,
 \qquad \ell\ge0.
$$
The continuation and the moment equations determine the numerical candidate; the proof uses the exact finite centre defined below.
\end{remark}

The normalised centre has
\begin{equation}\label{eq:k-centre}
 k^\circ=\frac{3861571936906315}{140737488355328}=27.438118883838420458\ldots
\end{equation}
and
\begin{equation}\label{eq:a-centre}
 a^\circ(z)=1+\sum_{j=1}^{20}a_j^\circ z^{13j}.
\end{equation}
The leading coefficients, displayed here as decimal approximations, are
\begin{equation}\label{eq:a-leading}
 \begin{aligned}
 a_1^\circ&\approx0.21111353950670197,&
 a_2^\circ&\approx0.15570614413673692,\\
 a_3^\circ&\approx0.040622815540330454,&
 a_4^\circ&\approx0.014382544407999130.
 \end{aligned}
\end{equation}
All centre coefficients, including the $861$ retained $g$ coefficients, are stored as exact hexadecimal binary64 numbers in the certificate. Appendix~\ref{app:coefficients} lists the conformal coefficients in decimal form. Choose
\begin{equation}\label{eq:finite-parameters}
 L=20,
 \qquad
 S=40,
 \qquad
 J_a=20.
\end{equation}
The retained radial indices are $s=0,\ldots,S$. The finite unknowns are
$$
 g_{\ell,s}\quad(0\le\ell\le L,\ 0\le s\le S),
 \qquad k,
 \qquad a_j\quad(1\le j\le J_a),
$$
and the finite residual consists of the matching $F_1$ coefficients and the boundary coefficients $(F_2)_\ell$, $0\le\ell\le J_a$. Hence
\begin{equation}\label{eq:finite-dimension}
 d_{\mathrm{fin}}=(L+1)(S+1)+J_a+1=882.
\end{equation}
Let $J_{\mathrm{fin}}$ be the derivative matrix in the normalised coordinates \eqref{eq:X-scaled}--\eqref{eq:Z-scaled}, and let $R$ be a fixed binary64 approximate inverse. Every entry of $R$ is treated as an exact dyadic rational.

On finite residuals, $\Aop$ acts by $R$. On the complement, an omitted $F_1$ coordinate is paired with the matching $g$ coordinate, while an omitted $F_2$ mode $\ell>J_a$ is paired with the matching shape coordinate with the sign that cancels the principal term $-(\dot a)_\ell$. These tail blocks are signed coordinate isometries.

\begin{lemma}[Injectivity of the approximate inverse]\label{lem:A-injective}
If
$
 \norm{I-RJ_{\mathrm{fin}}}_1<1,
$
then $R$ is invertible and the full operator $\Aop:\ZZ_\rho\to\XX_\rho$, including its two tail blocks, is a bounded isomorphism. Moreover,
\begin{equation}\label{eq:A-norm}
 \norm{\Aop}=\max\{\norm{R}_1,1\}.
\end{equation}
\end{lemma}

\begin{proof}
The direct-sum Neumann-series argument of \cite[Secs.~3.1 and~3.6]{ColbrookStepaniants2026} applies with the two signed coordinate-isometric tail blocks used here. It gives the invertibility of $R$ and $J_{\mathrm{fin}}$, the isomorphism property of $\Aop$, and the norm formula \eqref{eq:A-norm}.
\end{proof}

It remains to estimate $\Aop F(x^\circ)$ and $I-\Aop DF(x^\circ)$. Exact recurrences govern the finite components and reveal which tail columns can reach the retained block.

\subsection{Exact coefficient recurrences}

We use the exact disk-polynomial multiplication recurrences of \cite[eqs.~(40)--(42)]{ColbrookStepaniants2026}. In the present notation, for $q\in\mathbb Z$, $s\ge0$, and $n=|q|$, let
$$
 \Psi_{q,s}=r^nP_s^{(0,n)}(2r^2-1)e^{iq\theta},
$$
with $\Psi_{q,-1}=0$, so that $\Phi_{\ell,s}=\Psi_{m\ell,s}$. The cited rational identities give exact expansions under multiplication by $z$, $\overline z$, and $r^2$, independently of the imposed symmetry order. Their iteration expands every monomial $z^{mj_1}\overline z^{mj_2}$ in $|a|^2$. Together with \eqref{eq:KD-harmonic} and \eqref{eq:KD-nonharmonic}, they assemble $F(x^\circ)$ and $DF(x^\circ)$ entirely in the disk-polynomial basis.

\subsection{Tail estimates}\label{subsec:tails}

The omitted $g$-coordinates form the \emph{source tail}, and the omitted $a$-coordinates form the \emph{shape tail}. The recurrences settle every finite interaction. To control the infinitely many omitted coordinates, we isolate the columns that can meet the retained block and derive bounds, monotone in their indices, for all others.

\paragraph{Source tail.} The constants in the source-tail bound are
$$
 A_0=\norm{a^\circ}_{\PP_\rho},
 \qquad
 B_0=\norm{Ng^\circ}_{\BB_\rho},
 \qquad
 \zeta_\ell=\frac1{2(m\ell+1)}\quad(\ell\ge0).
$$
For a normalised real-symmetric source column $(\ell,s)$, the contribution remaining after the identity term satisfies
\begin{equation}\label{eq:g-tail-global}
 Z_g(\ell,s)
 \le (k^\circ)^2A_0^2\kappa_{\ell,s}
 +\mathbf{1}_{\{s=0\}}\,2\sigma B_0\zeta_\ell.
\end{equation}
The first term follows from Corollary~\ref{cor:coefficient-algebra} and \eqref{eq:kappa-column}; the second follows from the trace term $2b_gN\dot g$ in \eqref{eq:DF2}. The only additional contribution comes from its component in the retained residual block, where it is amplified by $R$.

\begin{lemma}[Source-tail support cutoff]\label{lem:g-tail-cutoff}
Assume that $a^\circ$ is supported in $0\le j\le J_a$, where $J_a=20$. A source column has zero component in every retained residual row, in either component of $F$, if
$$
 |\ell|>L+J_a
 \qquad\text{or}\qquad
 s\ge S+mJ_a+2.
$$
Thus there is no interaction with the retained residual block for $\ell\ge41$ or $s\ge302$.
\end{lemma}

\begin{proof}
For the $F_1$ component, a monomial in $|a^\circ|^2$ changes the symmetry index by at most $J_a$, proving the angular assertion. For the radial assertion, consider a term of radial degree $s'$ in $\KD\Phi_{\ell,s}$, so that $s'\ge s-1$, together with a monomial $z^{mj_1}\overline z^{mj_2}$, where $0\le j_1,j_2\le J_a$. Its output angular index is $\ell_{\mathrm{out}}=\ell+j_1-j_2$. With $t=r^2$, factor the source Jacobi weight $t^{m|\ell|}$ in the coefficient integral against a retained output polynomial of radial degree $p\le S$. The factor multiplying $P_{s'}^{(0,m|\ell|)}(2t-1)$ is $t^E P_p^{(0,m|\ell_{\mathrm{out}}|)}(2t-1)$, where
$$
 E=\frac{m(j_1+j_2)+m|\ell_{\mathrm{out}}|-m|\ell|}{2}.
$$
The triangle inequality gives $0\le E\le mJ_a$, and parity shows that $E$ is an integer. The multiplying polynomial has degree $p+E\le S+mJ_a$, so Jacobi orthogonality makes the coefficient vanish when $s'>S+mJ_a$. Since $s'\ge s-1$, the condition $s-1>S+mJ_a$ suffices.

For the $F_2$ component, let $b_g^\circ:=Ng^\circ$. By \eqref{eq:DF2}, the derivative in a source direction $\dot g$ is $2b_g^\circ N\dot g$. Formula~\eqref{eq:N-formula} gives $N\Phi_{\ell,s}=0$ for every $s\ge1$, so all such source columns have zero $F_2$ derivative. If $s=0$ and $|\ell|\ge41$, the trace of the corresponding real-symmetric source column is supported in boundary modes $\pm\ell$. Since $g^\circ$ has angular support at most $L=20$, the centre trace $b_g^\circ$ is supported in modes of absolute index at most $20$. Every mode in $2b_g^\circ N\dot g$ consequently has absolute index at least
$
 |\ell|-20\ge21.
$
It cannot meet a finite $F_2$ row, whose index has absolute value at most $J_a=20$. This proves the cutoff for both residual components.
\end{proof}

It follows that the only source-tail columns requiring direct evaluation on the retained residual block are
\begin{equation}\label{eq:g-near-set}
 \begin{aligned}
 &0\le\ell\le20,
 &&41\le s\le301,\\
 &21\le\ell\le40,
 &&0\le s\le301.
 \end{aligned}
\end{equation}
There are exactly $11{,}521$ such columns. Beyond this set, \eqref{eq:g-tail-global} is monotone. The entire angular tail is bounded by the two extremal columns $(41,0)$ and $(41,1)$, and the remaining radial tail by $(0,302)$.

\paragraph{Shape tail.} For a shape direction $\dot a=z^{mj}$, let
\begin{equation}\label{eq:Hj}
 H_j:=z^{mj}\overline{a^\circ}\,U^\circ\in\WW_\rho,
 \qquad U^\circ=\KD g^\circ.
\end{equation}
Formulas \eqref{eq:KD-harmonic} and \eqref{eq:KD-nonharmonic} show that $\KD$ preserves the angular index and changes the radial index by at most one. Hence
$$
 L_U:=L=20,
 \qquad S_U:=S+1=41,
$$
and $U^\circ$ is supported in $|\ell|\le L_U$, $0\le s\le S_U$. The corresponding derivative is
$$
 DF_1(x^\circ)[0,0,z^{mj}]
 =(k^\circ)^2(H_j+\overline{H_j}).
$$

\begin{lemma}[Shape-column support]\label{lem:shape-support}
Suppose that $a^\circ$ is supported in $0\le j_0\le J_a$, while $U^\circ$ is supported in $|\ell|\le L_U$ and $0\le s\le S_U$. A monomial contribution
$
 z^{mj}\overline z^{mj_0}\Phi_{\ell,s}
$
has output angular index
$
 \ell_{\mathrm{out}}=j-j_0+\ell
$
and output radial Jacobi degree at most $s+d_{\mathrm{sh}}$, where
\begin{equation}\label{eq:shape-d}
 d_{\mathrm{sh}}=\frac m2\bigl(|\ell|+j+j_0-|\ell_{\mathrm{out}}|\bigr)
 \le m(j_0+|\ell|)
 \le m(J_a+L_U).
\end{equation}
Consequently every shape column has radial support at most
$
 S_U+m(J_a+L_U).
$
For $J_a=L_U=20$, $S_U=41$, and $m=13$, the exact cutoff is $561$. For $21\le j\le60$, the angular support lies in $|\ell_{\mathrm{out}}|\le80$.
\end{lemma}

\begin{proof}
With $t=r^2$, the expression
$
t^{m(|\ell|+j+j_0)/2}P_s^{(0,m|\ell|)}(2t-1)
$
contains the output factor $t^{m|\ell_{\mathrm{out}}|/2}$. The remaining power is $t^{d_{\mathrm{sh}}}$, so the Jacobi expansion has degree at most $s+d_{\mathrm{sh}}$. Nonnegativity and integrality of $d_{\mathrm{sh}}$ follow from the triangle inequality and parity. Finally, we obtain
$$
 |\ell|+j+j_0-|\ell_{\mathrm{out}}|\le2(j_0+|\ell|),
$$
which proves \eqref{eq:shape-d}; the angular assertion is immediate from the support ranges.
\end{proof}

For $j\ge41$, the weighted-shift argument of \cite[Sec.~3.5]{ColbrookStepaniants2026} applies: every angular index in $H_j$ is positive and
$
 H_{j+1}=z^mH_j.
$
Therefore, we obtain the bound
\begin{equation}\label{eq:H-monotone}
 \rho^{-(j+1)}\norm{H_{j+1}}_\rho
 \le\rho^{-j}\norm{H_j}_\rho.
\end{equation}
For $j\ge61=J_a+L_U+L+1$, $H_j$ has no component in the retained residual block. The residual boundary coefficients after cancellation of the principal $-(\dot a)_j$ have exact normalised norm
\begin{equation}\label{eq:A-minus}
 A_-:=\sum_{j_0=1}^{J_a}|a_{j_0}^\circ|\rho^{-j_0}.
\end{equation}
Thus, we have
\begin{equation}\label{eq:shape-far}
 Z_a(j)
 \le\frac{(k^\circ)^2}{\sigma\rho^{41}}\norm{H_{41}}_\rho+A_-,
 \qquad j\ge61.
\end{equation}
The forty columns $21\le j\le60$ are evaluated directly over their exact support $|\ell|\le80$, $s\le561$; \eqref{eq:H-monotone}--\eqref{eq:shape-far} bound all remaining columns.

\paragraph{Residual and finite columns.} The centre residual has finite support. The shape has angular support $20$, while $U^\circ$ has angular support $20$ and radial support $41$. A monomial of $|a^\circ|^2U^\circ$ increases the radial Jacobi degree by at most
$$
 \frac m2\bigl(|\ell|+j_1+j_2-|\ell+j_1-j_2|\bigr)
 \le m\max\{j_1,j_2\}\le260.
$$
Hence $F_1(x^\circ)$ has angular support $40$ and radial degree at most $301$, and $F_2(x^\circ)$ has angular support $40$.

The finite $g$- and $k$-columns obey the same $F_1$ cutoff, while the finite shape columns are supported in $|\ell|\le40$, $s\le561$ by Lemma~\ref{lem:shape-support}. Thus the omitted part of every residual and every finite derivative column has finite support that can be enumerated explicitly. Together with \eqref{eq:g-near-set} and the shape partition above, these cases cover every residual row and every derivative column. These support partitions reduce the residual and derivative defect to finitely many interval evaluations and explicit monotone tail bounds.

\subsection{Directed-rounding validation}\label{subsec:directed-validation}

The recurrences are evaluated in the verified floating-point framework of~\cite{Rump2010}, using directed MPFR rounding~\cite{FousseEtAl2007} at two precision settings: twice at $256$ bits and once at $192$ bits with additional rounding guards. The two $256$-bit outputs agree byte for byte, and each corresponding $256$-bit enclosure is contained in the $192$-bit enclosure. The centre and the dense approximate inverse are exact dyadic inputs.

The finite matrix calculation applies the fixed matrix $R$ by a fixed-order binary64 fused multiply-add computation in the IEEE~754 model~\cite{IEEE7542019}; standard forward-error analysis encloses the result~\cite{Higham2002}. The forward-error constants are
$$
 u_{\mathrm{mach}}=2^{-53},
 \qquad
 \gamma_n=\frac{n\,u_{\mathrm{mach}}}{1-n\,u_{\mathrm{mach}}},
 \qquad
 \beta=2^{-1022},
 \qquad
 \delta_n=\frac{n\beta}{1-n\,u_{\mathrm{mach}}}.
$$
Suppose that $n\,u_{\mathrm{mach}}<1/2$. If $\Sigma_n=\sum_{i=1}^n v_iw_i$ with $v_iw_i\ge0$, and $\widehat\Sigma_n$ is the corresponding fixed-order binary64 FMA sum, the enclosure proved in \cite[Appendix~A.2]{ColbrookStepaniants2026}, in the present notation, is
\begin{equation}\label{eq:FMA-bound}
 \Sigma_n\le\frac{\widehat\Sigma_n+\delta_n}{1-\gamma_n}.
\end{equation}

For signed midpoint products, the computation separately encloses the sum of absolute products and propagates interval radii through $|R|$. The largest dot product has $882$ terms, so the condition $n\,u_{\mathrm{mach}}<1/2$ is automatic. The implementation rejects overflow, preserves exact zeros, and rounds subnormal midpoint-radius conversions outward; see Appendix~\ref{app:certificate}.

Combining the analytic majorants with these interval enclosures gives the quantitative conclusion from which the domain will be reconstructed.

\begin{theorem}[Existence and localisation]\label{thm:validated-zero}
With the parameters \eqref{eq:scales} and \eqref{eq:finite-parameters}, the operator $F$ has a unique zero $x_*=(g_*,k_*,a_*)$ in the closed $\XX$-ball of radius
$
 r=10^{-6}
$
about the exact dyadic centre $x^\circ$. The certified bounds include
\begin{align}
 Y&\le3.621873700919759\times10^{-9},
 \label{eq:Y-certified}\\
 Z&\le0.4732202748897342,
 \label{eq:Z-certified}\\
 \norm{\Aop}&\le642.8791637193662,
 \label{eq:A-certified}\\
 \norm{I-RJ_{\mathrm{fin}}}_1
 &\le4.367597503113863\times10^{-8}.
 \label{eq:inverse-defect-certified}
\end{align}
At $r=10^{-6}$,
\begin{align}
 p(r)&\le-4.890300069750544\times10^{-7}<0,
 \label{eq:p-certified}\\
 Z+\mathcal R'(r)&\le0.5414759642771006<1.
 \label{eq:contraction-certified}
\end{align}
Moreover,
\begin{align}
 k_*&\le27.438119883838421,
 \label{eq:k-upper-certified}\\
 \norm{a_*}_{\PP_\rho}&\le1.442055270128376,
 \label{eq:a-norm-certified}\\
 \norm{\KD g_*}_\rho&\le0.689578707902389,
 \label{eq:U-norm-certified}\\
 \Ree a_*(z)&\ge0.5727406866131367
 \qquad(|z|\le1),
 \label{eq:Re-a-certified}\\
 Ng_*(e^{i\theta})&\ge0.6058177124928927
 \qquad(\theta\in\R),
 \label{eq:trace-certified}\\
 |(a_*)_1|&\ge0.2111135329707542.
 \label{eq:a1-certified}
\end{align}
\end{theorem}

\begin{proof}
The radii criterion requires a residual bound, a bound for the defect of the full derivative including both tails, and injectivity of the approximate inverse. The interval computation supplies the first two in \eqref{eq:Y-certified}--\eqref{eq:Z-certified}; the finite inverse defect \eqref{eq:inverse-defect-certified} and Lemma~\ref{lem:A-injective} give the third. Equations \eqref{eq:p-certified}--\eqref{eq:contraction-certified} therefore yield the unique zero. The remaining estimates follow by interval evaluation over the coefficient ball and are recomputed in exact rational arithmetic by \texttt{verify\_certificate.py}. Appendix~\ref{app:certificate} gives the contribution to each bound and the data from which it is obtained.
\end{proof}

\section{Reconstruction of the domain and eigenfunction}\label{sec:reconstruction}

The validated zero $x_*=(g_*,k_*,a_*)$ supplied by Theorem~\ref{thm:validated-zero} lies in the coefficient space. Let $\phi_*$ satisfy $\phi_*'=a_*$ and $\phi_*(0)=0$. It remains to establish the geometry of its image, recover the unsquared boundary condition, continue the disc solution analytically through the boundary, and prove that the resulting eigenfunction changes sign.

\subsection{Univalence and quantitative geometry}

The centre satisfies
\begin{equation}\label{eq:centre-l1}
 \sum_{j=1}^{20}|a_j^\circ|<0.4272593067201966.
\end{equation}
The radius bound and the shape scaling give
\begin{equation}\label{eq:shape-ball}
 \norm{a_*-a^\circ}_{\PP_\rho}
 \le\frac{10^{-6}}{2\sigma}
 =\frac{10^{-6}}{150}.
\end{equation}
Combining \eqref{eq:centre-l1} and \eqref{eq:shape-ball} gives \eqref{eq:Re-a-certified}. The elementary Noshiro--Warschawski criterion used in \cite[Sec.~4.2]{ColbrookStepaniants2026} states that a holomorphic map on a convex domain is univalent whenever the real part of its derivative is strictly positive.

Since $\rho>1$, $a_*$ is holomorphic on $|z|<\rho^{1/13}$; the strict lower bound in \eqref{eq:Re-a-certified} persists on a slightly larger disc. Hence $\phi_*$ is univalent on a neighbourhood of $\overline\D$. With $\Omega:=\phi_*(\D)$, the boundary $\partial\Omega=\phi_*(\T)$ is a real-analytic Jordan curve.

The finite centre map is
\begin{equation}\label{eq:phi-centre}
 \phi^\circ(z)=z+\sum_{j=1}^{20}\frac{a_j^\circ}{13j+1}z^{13j+1}.
\end{equation}
For $|z|\le1$, \eqref{eq:shape-ball} gives the rigorous pointwise estimate
\begin{equation}
 |\phi_*(z)-\phi^\circ(z)|
 \le\sum_{j\ge1}\frac{|(a_*)_j-a_j^\circ|}{13j+1}
 \le\frac1{14\rho}\norm{a_*-a^\circ}_{\PP_\rho}
 \le4.6685340803\times10^{-10}.
 \label{eq:map-centre-error}
\end{equation}
Thus the curve in Figure~\ref{fig:domain} differs pointwise from the exact boundary by less than $4.67\times10^{-10}$.

The exact map $\phi_*$ has $D_{13}$ symmetry. The noncircularity argument of \cite[Sec.~4.2]{ColbrookStepaniants2026} applies: rotational symmetry would centre any disc at the origin, while uniqueness of the normalised Riemann map would then force $\phi_*$ to be the identity, contrary to \eqref{eq:a1-certified}. Hence $\Omega$ is not a disc.

The domain is also not centrally symmetric. A bounded centrally symmetric set has a unique centre, and the $13$-fold rotational symmetry would force that centre to be the origin. Central symmetry would then imply that
$
 \phi^\sharp(z):=-\phi_*(-z)
$
is another conformal map from $\D$ onto $\Omega$ satisfying $\phi^\sharp(0)=0$ and $(\phi^\sharp)'(0)=1$. Uniqueness of the normalised Riemann map would give $\phi^\sharp=\phi_*$, so $\phi_*$ would be odd. But the coefficient of $z^{14}$ in $\phi_*$ is $(a_*)_1/14$, which is nonzero by \eqref{eq:a1-certified}. This contradiction proves the claim.

\subsection{Recovery of the boundary condition}

The geometric argument leaves only the sign lost on squaring the boundary equation. Let $U_*=\KD g_*$. Proposition~\ref{prop:weak-KD} gives
$$
 \Delta U_*=g_*,
 \qquad
 U_*|_{\T}=0.
$$
The first fixed-disc equation gives
$$
 \Delta U_*+k_*^2|a_*|^2U_*=0.
$$
The second equation gives $(Ng_*)^2=|a_*|^2$. By \eqref{eq:trace-certified},
$$
 \inf_{\theta\in\R}Ng_*(e^{i\theta})
 \ge0.6058177124928927>0,
$$
while \eqref{eq:Re-a-certified} implies that $a_*$ has no zeros. Hence
$$
 Ng_*=|a_*|
 \quad\text{on }\T.
$$
Proposition~\ref{prop:fixed-disc-equivalence} therefore applies with $\varepsilon=1$, yielding
$
 u_*=U_*\circ\phi_*^{-1}
$
as a solution of \eqref{eq:main-pde}. The coefficient ball also gives
$
 |k_*-k^\circ|\le10^{-6}.
$
which together with \eqref{eq:k-centre}, implies the outward-rounded interval \eqref{eq:k-main}.

\subsection{Analytic regularity and sign change}

The analytic-continuation argument of \cite[Sec.~4.3]{ColbrookStepaniants2026} applies here with
$$
 c_*(x,y):=k_*^2|a_*(x+iy)|^2.
$$
Fix $q>2$. By Proposition~\ref{prop:weak-KD}, $U_*\in W^{2,q}(\D)\cap W^{1,q}_0(\D)$. The coefficient $c_*$ is real analytic on a neighbourhood of $\overline\D$, and interior analytic elliptic regularity gives $U_*\in C^\omega(\D)$~\cite{MorreyNirenberg1957}. Since $a_*$ has no zeros, $|a_*|$ is positive and real analytic on a collar of $\T$. Smooth boundary regularity gives $U_*\in C^\infty(\overline\D)$~\cite{AgmonDouglisNirenberg1959}. The Cauchy--Kowalevski theorem, applied in analytic boundary coordinates to
$$
 \widetilde U|_{\T}=0,
 \qquad
 \left.\partial_r\widetilde U\right|_{\T}=|a_*|,
 \qquad
 (\Delta+c_*)\widetilde U=0,
$$
produces a real-analytic solution on a two-sided collar of $\T$; see also~\cite{Millar1980}. The zero-extension and analytic-uniqueness step in \cite[Sec.~4.3]{ColbrookStepaniants2026}, now with these Cauchy data, identifies $\widetilde U$ with $U_*$ on the inner half-collar. Thus $U_*$ extends real analytically across $\T$.

Since $\phi_*$ is univalent on a complex neighbourhood of $\overline\D$, its inverse is holomorphic on a neighbourhood of $\overline\Omega$, after shrinking the neighbourhoods if necessary. Consequently
$$
 u_*=U_*\circ\phi_*^{-1}\in C^\omega(\overline\Omega).
$$

The eigenfunction $u_*$ must change sign. Otherwise it would be a first Dirichlet eigenfunction and would have a strict sign in the interior. After changing its sign if necessary, Serrin's theorem would force $\Omega$ to be a disc~\cite{Serrin1971}, contradicting the preceding noncircularity argument.

\begin{proof}[Proof of Theorem~\ref{thm:main}]
Theorem~\ref{thm:validated-zero} and the reconstruction above give a bounded simply connected $D_{13}$-invariant domain with real-analytic Jordan boundary, the frequency interval \eqref{eq:k-main}, and a real-analytic solution of \eqref{eq:main-pde}. The nonzero coefficient $(a_*)_1$ shows that the domain is neither a disc nor centrally symmetric, and Serrin's theorem gives the sign change. Proposition~\ref{prop:fourier-equivalence} gives \eqref{eq:boundary-fourier-main}. This is a counterexample to the planar Berenstein conjecture.
\end{proof}

\vspace{4mm}

\paragraph{\textbf{Acknowledgments.}} 
We thank the David Crighton Fund for providing a fellowship for SS to visit MJC in Cambridge. We also thank Miles Wheeler for discussions regarding~\cite{Wheeler2025}. GS is supported by an
NSF Mathematical Sciences Postdoctoral Research Fellowship (MSPRF) under award number 2402074.

\vspace{4mm}

\paragraph{\textbf{AI declaration.}}
We believe it is important to declare the use of AI in mathematical research, and in the present case its role is particularly noteworthy. ChatGPT 5.6 was given a substantial warm start consisting of an early draft and working code independently developed by MJC and GS for \cite{ColbrookStepaniants2026}. That material already contained the central construction, conformal
fixed-disc formulation, coefficient spaces, disk-polynomial algebra,
tail strategy, and computational architecture on which the present
paper rests. The system was used to explore the modifications required
at the complementary Dirichlet endpoint and suggested ideas contributing
to early versions of Lemma~\ref{lem:KD-exact} and
Proposition~\ref{prop:fixed-disc-equivalence}, as well as to a revised
version of the code. We subsequently developed, refined, and
independently checked all resulting arguments and computations. The
authors take full responsibility for all content.

\appendix

\section{The computational certificate}\label{app:certificate}

This appendix specifies the exact data and the independent checks underlying the finite bounds in Theorem~\ref{thm:validated-zero}.

\subsection{Archive, exact inputs, and reproduction}

The supplementary computational certificate is contained in the GitHub repository directory
\begin{center}
\texttt{berenstein\_validation\_certificate/}.
\end{center}
The relevant SHA-256 digests are listed below. File paths are relative to the root of the certificate package.
\begin{center}
\small
\begin{tabular}{@{}l@{}}
\toprule
\texttt{berenstein\_validated\_certificate.zip}\\
\texttt{ed616db715a5f86256cd1b2d0b4db73f50d280eb68e2fbe08a3d3c2da7001aee}\\
\addlinespace[.5em]

\texttt{data/center\_L20\_S40\_J20.hex}\\
\texttt{6edbcf1b90ffcab5a7d198fd7f5d76199715e81f52d98f675d830fd5cd3187a4}\\
\addlinespace[.5em]

\texttt{data/approx\_inverse\_sigma75.bin}\\
\texttt{da18ec4b0622f80ea1b1192ecc28b0f27ec53d51ec6cf11b1484e9ec4f339d0b}\\
\addlinespace[.5em]

\texttt{source\_manifest.txt}\\
\texttt{f7b820b0307c5251bab14d6b80c6bd47da8c78acfaf4264f060df2d1295eded9}\\
\addlinespace[.5em]

\texttt{audit\_mpfr\_192.json}\\
\texttt{f88ae47a4602b6111d2b15a29e1b3790c57093ed9b677ab5003e19b7d023b205}\\
\addlinespace[.5em]

\texttt{audit\_mpfr\_256a.json}\\
\texttt{ae5f96cda2c16b2e3b1fd5f2836946eecd2c7aa23a3f2fb7475ffdca96ebf9b5}\\
\addlinespace[.5em]

\texttt{audit\_mpfr\_256b.json}\\
\texttt{ae5f96cda2c16b2e3b1fd5f2836946eecd2c7aa23a3f2fb7475ffdca96ebf9b5}\\
\addlinespace[.5em]

\texttt{certificate.json}\\
\texttt{f43bf6525b51cdfff26599a81d34b12fbec7da285dd99ab770c418076a286cca}\\
\addlinespace[.5em]

\texttt{SHA256SUMS}\\
\texttt{a5eaf48cf4e6b85ae4b0970ec44e3a227d7cfe66851eab2571d04ace07f8e98f}\\
\bottomrule
\end{tabular}
\end{center}
The internal file \texttt{SHA256SUMS} lists the SHA-256 digest of every other file in the certificate package. The standalone verifier checks that the manifest, the expected package paths, and the files present agree exactly.

\paragraph{Reproduction.} After extraction into a new directory, the command
\begin{verbatim}
THREADS=24 ./reproduce.sh
\end{verbatim}
performs the directed-rounding computation. The reference execution used GCC~11.5.0 (Red Hat 11.5.0-14, 20240719) and MPFR~4.2.1 with thread-local state, explicit RNDD/RNDU rounding, and $24$ threads. It produced one $192$-bit run with additional rounding guards and two $256$-bit runs. The two $256$-bit output files agree byte for byte, and their intervals are contained in those of the guarded $192$-bit run.

\paragraph{Independent verification.} Exact rational aggregation is verified by
\begin{verbatim}
python3 verify_certificate.py certificate.json
\end{verbatim}
After checking the package and source manifests, the input and audit identities, support counts, interval nestings, radii inequalities, and geometric inequalities, the verifier returns the terminal message
\begin{verbatim}
AUTHENTICATED BERENSTEIN D13 CERTIFICATE: PROVED
\end{verbatim}

The directed-rounding computation produces the interval enclosures, which the checker aggregates exactly. Combined with the analytic estimates above, these calculations prove the bounds in Theorem~\ref{thm:validated-zero}.

\subsection{Certified decomposition}

Table~\ref{tab:certificate-bounds} contains the bounds from either of the two byte-identical $256$-bit runs.

\begin{table}[t]
\centering
\small
\caption{Certified bounds used in the contraction and reconstruction arguments.}
\label{tab:certificate-bounds}
\begin{tabular}{@{}>{\raggedright\arraybackslash}p{.61\textwidth}>{\raggedleft\arraybackslash}p{.31\textwidth}@{}}
\toprule
Quantity & Certified bound\\
\midrule
Finite residual after the fixed inverse & $3.082024185794305\times10^{-9}$\\
Omitted $F_1$ residual & $2.471737072026991\times10^{-10}$\\
Omitted $F_2$ residual & $2.926758079227554\times10^{-10}$\\
Global residual bound $Y$ & $3.621873700919759\times10^{-9}$\\
Finite inverse defect $\norm{I-RJ_{\mathrm{fin}}}_1$ & $4.367597503113863\times10^{-8}$\\
Largest finite $g$-column defect & $0.06337480546348755$\\
Finite $k$-column defect & $3.720437321313102\times10^{-9}$\\
Largest finite shape-column defect & $0.00937648688769097$\\
Largest source-tail column meeting the retained block & $0.4593799524396485$\\
Remaining angular source tail, $s=0$ & $0.2050085510518364$\\
Largest shape column, $21\le j\le40$ & $0.4732202748897342$\\
Largest shape column, $41\le j\le60$ & $0.4506521964853179$\\
Remaining shape tail, $j\ge61$ & $0.4506521964853019$\\
Global derivative defect $Z$ & $0.4732202748897342$\\
Approximate-inverse norm $\norm{\Aop}$ & $642.8791637193662$\\
Radii polynomial $p(10^{-6})$ & $-4.890300069750544\times10^{-7}$\\
Contraction bound & $0.5414759642771006$\\
Lower bound for $\Ree a_*$ & $0.5727406866131367$\\
Lower bound for $Ng_*$ & $0.6058177124928927$\\
Lower bound for $|(a_*)_1|$ & $0.2111135329707542$\\
\bottomrule
\end{tabular}
\end{table}

The support calculation gives the following counts and cutoffs:
$$
 \begin{array}{c|c}
 \text{part of the calculation}&\text{enumeration or cutoff}\\ \hline
 \text{source-tail columns evaluated directly}&11{,}521\text{ columns}\\
 \text{remaining angular source columns}&\ell\ge41\\
 \text{remaining radial source columns}&s\ge302\\
 \text{shape-tail columns evaluated directly}&21\le j\le60\\
 \text{shape angular support}&|\ell|\le80\\
 \text{shape radial support}&s\le561\\
 \text{remaining shape columns}&j\ge61.
 \end{array}
$$
At the centre, the $F_1$ residual support is contained in $|\ell|\le40$, $s\le301$, and the $F_2$ residual support in $|\ell|\le40$.

\subsection{Analytic and computational inputs}

The following table gives the justification for each ingredient of the proof.

\begin{center}
\small
\begin{tabular}{@{}>{\raggedright\arraybackslash}p{.30\textwidth}>{\raggedright\arraybackslash}p{.61\textwidth}@{}}
\toprule
Ingredient & Justification\\
\midrule
Disk-polynomial identities & Lemma~\ref{lem:KD-exact}, the recurrences in \cite[Sec.~3.3]{ColbrookStepaniants2026}, and exact rational-polynomial tests in \texttt{tests/symbolic.py}.\\
Finite residual and Jacobian & MPFR recurrence evaluation with directed rounding in \texttt{src/interval\_assemble.cpp}; the fixed dyadic inverse is enclosed by \eqref{eq:FMA-bound}.\\
Finite inverse defect & Complete $882\times882$ column-sum enclosure with fixed ordering.\\
Omitted residual and finite-column rows & Complete recurrence-generated support lists, reconstructed and checked independently.\\
Source tail & The $11{,}521$ columns in \eqref{eq:g-near-set}, followed by \eqref{eq:g-tail-global} and monotonicity.\\
Shape tail & Forty columns through the exact support $|\ell|\le80$, $s\le561$, followed by \eqref{eq:H-monotone}--\eqref{eq:shape-far}.\\
Nonlinearity & Exact polynomial majorant \eqref{eq:M1}--\eqref{eq:R-majorant}, algebra constant one, and the sharp norms \eqref{eq:sharp-KD-N}.\\
Radii inequalities & Exact rational recomputation by \texttt{verify\_certificate.py}.\\
Geometry & Coefficient-ball estimates \eqref{eq:shape-ball}, \eqref{eq:map-centre-error}, and the certified trace and first-shape bounds.\\
Reproducibility & Two byte-identical $256$-bit runs, with intervals contained in those of the guarded $192$-bit run.\\
\bottomrule
\end{tabular}
\end{center}

\section{Conformal coefficients of the numerical centre}\label{app:coefficients}

The conformal-map coefficients $\phi_j^\circ:=a_j^\circ/(13j+1)$ in
$$
 \phi^\circ(z)=z+\sum_{j=1}^{20}\phi_j^\circ z^{13j+1}
$$
are listed below in decimal form; the certificate uses the exact hexadecimal centre coefficients $a_j^\circ$.

\begin{longtable}{@{}r@{\qquad}r@{\qquad}r@{\qquad}r@{}}
\toprule
$j$ & $\phi_j^\circ$ & $j$ & $\phi_j^\circ$\\
\midrule
\endfirsthead
\toprule
$j$ & $\phi_j^\circ$ & $j$ & $\phi_j^\circ$\\
\midrule
\endhead
1 & $1.50795385361929982\times10^{-2}$ & 11 & $8.99107926650920794\times10^{-9}$\\
2 & $5.76689422728655252\times10^{-3}$ & 12 & $2.06853696471111575\times10^{-9}$\\
3 & $1.01557038850826134\times10^{-3}$ & 13 & $4.75778692314195368\times10^{-10}$\\
4 & $2.71368762415077922\times10^{-4}$ & 14 & $1.09440420794361152\times10^{-10}$\\
5 & $5.94916615455076681\times10^{-5}$ & 15 & $2.51724454756957816\times10^{-11}$\\
6 & $1.40376217202844726\times10^{-5}$ & 16 & $5.78995922871089628\times10^{-12}$\\
7 & $3.19906904482506334\times10^{-6}$ & 17 & $1.33169815228079109\times10^{-12}$\\
8 & $7.39433442683568780\times10^{-7}$ & 18 & $3.06233427501141624\times10^{-13}$\\
9 & $1.69810772559526110\times10^{-7}$ & 19 & $7.09000044668921563\times10^{-14}$\\
10 & $3.90984465868492322\times10^{-8}$ & 20 & $1.71384604459278807\times10^{-14}$\\
\bottomrule
\end{longtable}

\section*{Data and code availability}

The computational artefact accompanying this paper, including the independently executable verifier, authenticated source, exact inputs, directed-rounding
audit files, certificate, and reproduction instructions, is available at
\url{https://github.com/sgstepaniants/Berenstein}.

The GitHub repository contains the unpacked certificate package under
\begin{center}
\texttt{berenstein\_validation\_certificate/},
\end{center}
together with \path{REPRODUCTION_REPORT.md} and
\path{PROOF_OVERVIEW.md}.

For archival distribution, the certificate package is available at Zenodo under DOI \texttt{10.5281/zenodo.21865020}. Appendix~\ref{app:certificate} gives the exact inputs, cryptographic identities, reproduction command, and independent verification procedure.

\small
\bibliographystyle{abbrv}
\bibliography{berenstein_counterexample_refs}

\end{document}